\documentclass[11pt, reqno]{amsart}

\usepackage{amssymb,latexsym,amsmath,amsfonts}
\usepackage{mathrsfs}
\usepackage{graphicx}
\usepackage[usenames]{color}
\usepackage{hyperref}
\usepackage{comment}
\usepackage{enumitem}

\definecolor{DPurple}{rgb}{0.46,0.2,0.69}

\numberwithin{equation}{section}

\allowdisplaybreaks

\theoremstyle{definition}
\newtheorem{definition}{Definition}[section]

\theoremstyle{remark}

 \theoremstyle{plain}
\newtheorem{theorem}{Theorem}[section]
\newtheorem{result}{Result}[section]

\newtheorem{lemma}[definition]{Lemma}



\begin{document}

\title[On $\varphi$-Normality of Harmonic Mappings and Their Families]{On $\varphi$-Normality of Harmonic Mappings and Their Families}

\author{Gopal Datt}
\address{Department of Mathematics, Babasaheb Bhimrao Ambedkar University, Lucknow, India}
\email{ggopal.datt@gmail.com, gopal.du@gmail.com}
\author{Ritesh Pal}
\address{Department of Mathematics, Babasaheb Bhimrao Ambedkar University, Lucknow, India}
\email{rriteshpal@gmail.com, ritesh.rs.math@bbau.ac.in}
%\author{Ashish Kumar Trivedi}
%address{ Department of Mathematics, University of Delhi, Delhi 110007}
%\email{trivediashish2016@gmail.com, aktrivedi@maths.du.ac.in}


\begin{abstract}
In this paper, we present Lappan's five-point type criteria for 
$\varphi$-normality of sense-preserving harmonic mappings
and families of harmonic mappings in the unit disc $\mathbb{D}$.
We establish a normality criterion involving the 
higher-order spherical derivative of order $k$, and $k+2$ distinct 
test values under suitable conditions. We further reduce the 
number of test values to $\left[\frac{k}{2}\right]+2$ under an 
additional  differential condition involving the analytic and co-analytic 
parts. We also prove a generalization in which the 
spherical derivative is replaced by a polynomial with non-negative
real coefficients. 
\end{abstract}

\maketitle
\section{Introduction}

The theory of normal families is one of the fundamental
topics in complex analysis and has  important  applications in several 
areas including geometric function theory,
value distribution theory, and complex dynamics. In the early
twentieth century, Montel~\cite{Montel1912} introduced one of the 
most fundamental principles of compactness, namely {\it normal 
family} of meromorphic functions. The notion
of {\it normal family} provides a powerful framework
for studying the limiting behaviour of sequences of meromorphic
functions.
\smallskip

To make this notion precise, we first recall the definition of 
normal family.
Let $D\subset \mathbb{C}$ be a domain. A family of meromorphic 
functions on $D$  is said to be
\textit{normal} if every sequence in the family has a 
subsequence which converges locally uniformly with 
respect to the spherical metric, to a meromorphic
function or to infinity.
One of the most useful characterizations in the normal
families is provided by Marty's criterion, which asserts that,
a family of meromorphic functions is {\it normal} if and only if the
spherical derivative of the functions of the family is
locally uniformly bounded; recall that the spherical derivative
 of a meromorphic function $f$
is 
\begin{equation}\label{eq:sphder}
f^\#(z)=\frac{|f'(z)|}{1+|f(z)|^2}.
\end{equation}  
\smallskip 

The notion of normality can also be associated with an individual
meromorphic function. In this setting,
the idea of a normal function was first introduced 
by Yosida~\cite{Yosida1934} and
subsequently investigated by Noshiro~\cite{Noshiro1938}, although
they did not originally use the term {\it normal function}.
The terminology {\it normal function} was later introduced 
by Lehto and Virtanen~\cite{LehtoVirtanen1957}. In particular, let $D\subset\mathbb{C}$
 be a domain and let $f$ be a meromorphic function in
 domain $D$. The function $f$ is said to be
 {\bf normal} if the family $\mathcal{F}=\{f\circ\phi: \phi \in Aut{(D)}\}$
 is a normal family in domain $D$ in the sense of Montel.
 \smallskip
 
 The definition of normal meromorphic functions can be 
 characterized in terms of the spherical derivative as follows.

\begin{definition}\cite{LehtoVirtanen1957}
  A meromorphic function $f$ in the unit disc 
  $\mathbb{D}=\{z\in \mathbb{C}:|z|<1\}$ is
  called normal if
  \[
  \sup_{z\in\mathbb{D}}(1-|z|^2)f^\#(z)<\infty,
  \] 
  where $f^\#$ is the spherical derivative as in~\eqref{eq:sphder}.
\end{definition}

The theory of normal functions, initially developed
in the setting of meromorphic functions, has also been 
extended to complex-valued harmonic mappings. This extension is natural
because harmonic mappings play an important role in
geometric function theory.
Let $D\subset \mathbb{C}$ be a simply connected domain, 
a complex-valued harmonic mapping $f$
in $D$ can be represented in the form 
$f=h+\overline{g}$, where $h$ and $g$
are analytic functions in $D$ with $g(z_0)=0$ for some given
point $z_0\in D$. Its Jacobian is given by 
$J_f(z)=|h'(z)|^2-|g'(z)|^2$. Lewy~\cite{Lewy1936}
proved that
a harmonic mapping is locally univalent if and only 
if $J_f(z)\neq0$ in $D$. Moreover,
$f$ is sense-preserving, whenever $J_f(z)>0$ 
throughout the domain $D$.
\smallskip

 Motivated by Colonna's work~\cite{colona} on
  Bloch harmonic functions, Arbel\'aez,
 Hern\'andez, and Sierra~\cite{Arbelaez2019} 
 introduced the notion of normal harmonic mapping in the 
 unit disc $\mathbb{D}=\{z\in \mathbb{C}: |z|<1\}$.
 They obtained a quantitative characterization of normal harmonic 
 mapping in terms of the spherical derivative as follows.
  
\begin{definition}\cite[Proposition 2.1]{Arbelaez2019}
  A harmonic mapping $f=h+\overline{g}$ in $\mathbb{D}$ 
  is normal in $\mathbb{D}$ if 
  \[
  \sup_{z\in\mathbb{D}}(1-|z|^2)f^\#(z)<\infty
  \] 
  
\end{definition}

where
  \begin{equation}\label{eq:sphderhar}
    f^\#(z)=\frac{|h'(z)|+|g'(z)|}{1+|f(z)|^2}.
  \end{equation}

In analogy with the higher-order spherical derivatives for 
meromorphic functions formulated by Lappan~\cite{Lappan1977}, Bharti 
and Thin~\cite{bharatithin} introduced a corresponding
higher-order spherical derivative for harmonic mappings.
The extended (or higher-order) spherical derivative of order $k$ 
 of a harmonic mapping $f=h+\overline{g}$  is defined as follows~\cite[p. 389]{bharatithin}: 
\begin{equation}\label{eq:Ksphderhar}
    f^{\#(k)}(z): =\frac{|h^{(k)}(z)|+|g^{(k)}(z)|}{1+|f(z)|^{k+1}}.
  \end{equation}

In this paper, we study the notion of $\varphi$-normality for harmonic mappings and 
families of harmonic mappings. The notion of  $\varphi$-normality for meromorphic 
functions was introduced by Aulaskari, Makhmutov, and R\"aty\"a in \cite{Aulaskari}. 
Let us first recall the notion of a smoothly increasing function introduced in the study of
 $\varphi$-normal meromorphic function. Let $\varphi:[0,1)\to (0,\infty)$ be a 
continuous increasing function. We say $\varphi$ is smoothly increasing if
\[
\varphi(r)(1-r)\rightarrow \infty,\qquad \text{as}\,  (r\to1^-),
\] 
and
\[
\mathcal{R}_a(z)=\frac{\varphi\left(\left|a+\frac{z}{\varphi(|a|)}
\right|\right)}{\varphi(|a|)}\rightarrow 1, \qquad \text{as}\, 
(|a|\to 1^-),
\]
converges uniformly on compact subsets of $\mathbb{C}$ with,
 \[
\varphi(r)(1-r)\geq1 \quad 0\leq r <1.
\]

For such a function $\varphi$, a meromorphic 
  function in $\mathbb{D}$ is called {\it $\varphi$-normal} if
  \[
  \sup_{z\in\mathbb{D}}\frac{f^\#(z)}{\varphi(|z|)}<\infty,
  \]
  where $f^\#$ denotes the spherical derivative of $f$ as given in ~\eqref{eq:sphder}.
  \smallskip
  
  Following this framework, the authors, together with Bohra~\cite{Bohra},
  extended the notion of 
  $\varphi$-normality to harmonic mapping as follows.
\begin{definition}\cite[Definition 1.2]{Bohra}
    If $f=h+\overline{g}$ is
    a harmonic mapping in $\mathbb{D}$,
      then $f$ is said to be
   {\bf $\varphi$-normal} if
  \[
  \sup_{z\in\mathbb{D}}\frac{f^\#(z)}{\varphi(|z|)}<\infty,
  \]
  where the spherical derivative $f^\#$ is as defined in \eqref{eq:sphderhar}.

\end{definition}

The notion of $\varphi$-normality was first formulated for
meromorphic functions by Aulaskari et al. in~\cite{Aulaskari} and was subsequently 
extended to individual
harmonic mappings by Bohra et al. in~\cite{Bohra}. More recently, Charak et 
al.~\cite{charakbharti} extended
this framework further by introducing $\varphi$-normal families of
harmonic mappings.

\begin{definition}\cite[Definition 1.9]{charakbharti}
A family $\mathcal{F}$ of harmonic mappings in 
$\mathbb{D}$ is said to be
{\bf $\varphi$-normal} if for each compact set
$K\subset\mathbb{D}$ there exists $M>0$ such that
\[
\sup\left\{
\frac{f^{\#}(z)}{\varphi(|z|)}
:\;
z\in K,\;
f\in\mathcal{F}
\right\}
<M.
\]
\end{definition}

Bharti and Thin~\cite{bharatithin} applied the extended spherical 
derivative to derive a sufficient condition for  $\varphi$-normality 
of harmonic mappings in $\mathbb{D}$:
\begin{result}\label{Bharatik2}\cite[Theorem 1.8]{bharatithin}
  Let $k$ be a positive integer and let $f=h+\overline{g}$ be a 
  sense-preserving harmonic mapping in $\mathbb{D}$ such that
  \[\sup\left\{
|f^{(i)}(z)|:
z\in f^{-1}(\{0\}), \;
i=0,\ldots,k-1
\right\}<\infty.\]

If there exists a set $E$ of $k+4$ distinct points in $\mathbb{C}$
such that \[\sup\left\{
\frac{f^{\#(k)}(z)}{\varphi(|z|)^k}
:
z\in f^{-1}(E)
\right\}
<\infty,\] then $f$ is a $\varphi$-normal harmonic mapping.
\end{result} 
We refer to the elements of $E$ as test values and to $E$ as the test set.
\smallskip

Lappan~\cite[Theorem 1]{Lappan1974} established a five-point criterion for the 
normality of meromorphic functions, showing that it suffices to control the
 spherical derivative on the preimages of five distinct points. Subsequently,
 Bohra et al.~\cite[Theorem 5]{Bohra} obtained a Lappan-type criterion for 
$\varphi$-normal harmonic mappings with the same five-point test values. More recently, 
Charak et al.~\cite{charakbharti} reduced the number of test points to three 
for individual sense-preserving harmonic mappings.

\begin{result}\label{threevaluefunction}\cite[Theorem 1.13]{charakbharti}
Let $\varphi:[0,1)\to(0,\infty)$ be a continuous smoothly increasing
function. Suppose $f=h+\overline{g}$ is a non-constant sense-preserving
 harmonic mapping in $\mathbb{D}$ and $E$
  be a set of three distinct points in $\mathbb{C}$ such that
\[
\sup\left\{
\frac{f^{\#}(z)}{\varphi(|z|)}
:\;
z\in f^{-1}(E)
\right\}
<\infty.
\]

Then $f$ is a $\varphi$-normal harmonic mapping 
in $\mathbb{D}$.
\end{result}

The Results~\ref{Bharatik2} and~\ref{threevaluefunction} naturally 
lead to the following question:\smallskip

$\bullet$ {\it Can the cardinality of the test set $E$
 in Result~\ref{Bharatik2} be reduced from $k+4$ 
to $k+2$ test values?} \smallskip

We answer this question affirmatively 
by establishing a corresponding criterion involving the extended spherical
 derivative of order $k$.

\begin{theorem}\label{Mainfunction}
Let $\varphi:[0,1)\to(0,\infty)$ be a continuous smoothly 
increasing
function. Suppose
$f=h+\overline{g}$ is a sense-preserving harmonic
mapping in $\mathbb{D}$, and $k$ is a positive integer such that
\begin{equation}\label{eq:1MainfunHyp}
\sup\left\{
|f^{(i)}(z)|:
z\in f^{-1}(\{0\}),\;
i=0,\ldots,k-1
\right\}<\infty.
\end{equation}

If there exists a set $E$ of $k+2$ distinct 
points in $\mathbb{C}$ such that
\begin{equation}\label{eq:2MainfunHyp}
\sup\left\{
\frac{f^{\#(k)}(z)}{\varphi(|z|)^k}
:
z\in f^{-1}(E)
\right\}
<\infty,
\end{equation}
then $f$ is a $\varphi$-normal harmonic mapping, 
where $f^{\#(k)}$ is as in~\eqref{eq:Ksphderhar}.
\end{theorem}
\medskip

For $k=1$, Theorem~\ref{Mainfunction} reduces to
Result~\ref{threevaluefunction}, thus Theorem~\ref{Mainfunction}
may be viewed as a higher-order spherical derivative extension of the 
three-point criterion by Charak et al.~\cite{charakbharti}.
 \smallskip
 
 Charak et al.~\cite{charakbharti} also established a Lappan's five-point type 
 criterion for $\varphi$-normal families of sense-preserving harmonic mappings,
 as follows.
 
\begin{result}\label{threevalue}\cite[Theorem 1.12]{charakbharti}
Let $\varphi:[0,1)\to(0,\infty)$ be a continuous smoothly increasing
function. Suppose $\mathcal{F}$ is a family of sense-preserving harmonic
mappings in $\mathbb{D}$ and $E$ is a set of three distinct
complex numbers such that
\[
\sup\left\{
\frac{f^{\#}(z)}{\varphi(|z|)}
:\;
z\in f^{-1}(E),\;
f\in\mathcal{F}
\right\}
<\infty.
\]

Then $\mathcal{F}$ is $\varphi$-normal family of harmonic mappings in $\mathbb{D}$.
\end{result}

This result, together with the preceding Theorem~\ref{Mainfunction} for
$\varphi$-normal harmonic mappings 
 naturally suggests investigating {\it whether an analogous
normality criterion can be obtained for families of harmonic mappings 
involving higher-order spherical derivatives.}
\smallskip

To address this question, we establish an analogue of 
Theorem~\ref{Mainfunction} for families of harmonic mappings
by incorporating the higher-order spherical derivative.

\begin{theorem}\label{mainfamily}
Let $\varphi:[0,1)\to(0,\infty)$ be a continuous smoothly increasing
function and $k$ be a positive integer. Suppose $\mathcal{F}$
 is a family of sense-preserving harmonic
mappings in $\mathbb{D}$ such that
\begin{equation}\label{eq:1mainfamily}
\sup\left\{
|f^{(i)}(z)|:
z\in f^{-1}(\{0\}), f\in \mathcal{F},\;
i=0,\ldots,k-1
\right\}<\infty.
\end{equation}

If there exists a set $E$ of $k+2$ distinct 
points in $\mathbb{C}$ such that
\begin{equation}\label{eq:2mainfamily}
\sup\left\{
\frac{f^{\#(k)}(z)}{\varphi(|z|)^k}
:
z\in f^{-1}(E), f\in \mathcal{F}
\right\}
<\infty,
\end{equation}
then $\mathcal{F}$ is a $\varphi$-normal family.
\end{theorem}
 \medskip
For $k=1$, Theorem~\ref{mainfamily} reduces to 
Result~\ref{threevalue}, which is therefore obtained  
as a special case of Theorem~\ref{mainfamily}.
\smallskip

The four-point criterion established by Bohra et al.~\cite[Theorem 1.6]{Bohra}
motivates the following problem: {\it can the number 
of test values in Result~\ref{threevaluefunction}
 be further reduced under an additional differential condition?}
 We show that two test values suffice when a suitable boundedness
condition is imposed on a differential expression involving the analytic 
and co-analytic parts of harmonic mapping. This leads to the
following two-point refinement of Result~\ref{threevaluefunction}.
\begin{theorem}\label{twovaluefunction}
  Let $\varphi$ be a smoothly increasing function. 
  Suppose $f$ is a 
   sense-preserving harmonic mapping in $\mathbb{D}$.
    If there exists a set $E$ of $2$ 
   distinct points in $\mathbb{C}$ such that
\begin{equation}\label{eq:1twovaluefunction}
\sup\left\{
\frac{f^\#(z)}{\varphi(|z|)}
:
z\in f^{-1}(E)
\right\}
<\infty,
\end{equation}
and
\begin{equation}\label{eq:2twovaluefunction}
\sup\left\{
\frac{|h''(z)|+|g''(z)|}
{1+\left(|h'(z)|+|g'(z)|\right)^2}
:
z\in f^{-1}(E)
\right\}
<\infty,
\end{equation}
then $f$ is a $\varphi$-normal harmonic mapping.
\end{theorem}

In Theorem~\ref{twovaluefunction}, a two-point criterion is 
 established for an individual sense-preserving harmonic mapping. 
This motivates investigating {\it whether the same reduction to two 
test values can be achieved for families of harmonic mappings in 
 Result~\ref{threevalue}.} We show that, under the corresponding 
uniform differential conditions, a two-point criterion can be 
established for families of harmonic mappings.
 \begin{theorem}\label{twovaluefamily}
  Let $\varphi$ be a smoothly increasing function. 
  Suppose $\mathcal{F}$ is
   a family of sense-preserving harmonic mappings in 
   $\mathbb{D}$.
If there exists a set $E$ of $2$ distinct points 
in $\mathbb{C}$ such that
\begin{equation}\label{eq:1twovaluefamily}
\sup\left\{
\frac{f^\#(z)}{\varphi(|z|)}
:
z\in f^{-1}(E), \, f \in \mathcal{F}
\right\}
<\infty,
\end{equation}
and
\begin{equation}\label{eq:2twovaluefamily}
\sup\left\{
\frac{|h''(z)|+|g''(z)|}
{1+\left(|h'(z)|+|g'(z)|\right)^2}
:
z\in f^{-1}(E), \, f \in \mathcal{F}
\right\}
<\infty,
\end{equation}
then $\mathcal{F}$ is a $\varphi$-normal family.
\end{theorem}

Bharti and Thin~\cite[Thoerem 1.10]{bharatithin} established that, 
under suitable conditions, the test values in Result~\ref{Bharatik2} 
can be reduced from $k+4$ to $[k/2]+4$,  where $[x]$ 
denotes the greatest integer less than or equal to $x$.
 \begin{result}\label{bharti3}\cite[Theorem 1.10]{bharatithin}
   Let $\varphi$ be a smoothly increasing function, and
$k$ be a positive integer. Suppose $\mathcal{F}$ is a family 
of sense-preserving harmonic mappings in $\mathbb{D}$ such that
\begin{equation}
\sup\left\{
|f^{(i)}(z)|:
z\in f^{-1}(\{0\}),\, f \in \mathcal{F}, \;
i=0,\ldots,k-1
\right\}<\infty.
\end{equation}

If there exists a set $E$ of $\left[\frac{k}{2}\right]+4$
 distinct points in $\mathbb{C}$ such that
\begin{equation}
\sup\left\{
\frac{f^{\#(k)}(z)}{\varphi(|z|)^k}
:
z\in f^{-1}(E), \, f \in \mathcal{F}
\right\}
<\infty,
\end{equation}
and
\begin{equation}
\sup\left\{
\frac{|h^{(k+1)}(z)|+|g^{(k+1)}(z)|}
{1+\left(|h^{(k)}(z)|+|g^{(k)}(z)|\right)^{k+1}}
:
z\in f^{-1}(E)
\right\}
<\infty,
\end{equation}
then $f$ is a $\varphi$-normal harmonic mapping.
 \end{result}
 
The Results~\ref{Bharatik2} and~\ref{bharti3} suggest the following 
problem: \smallskip

$\bullet$ {\it Can the cardinality of the test set $E$ in 
Theorem~\ref{Mainfunction} be further reduced by imposing 
an additional differential condition?}
Motivated by this question, we establish the following 
criterion, in which the number of test values is 
reduced from $k+2$ to $[k/2]+2$.

\begin{theorem}\label{K: twovaluefunction}
Let $\varphi$ be a smoothly increasing function, and
$k$ be a positive integer. Suppose $f$ is sense-preserving 
harmonic mapping in $\mathbb{D}$ such that
\begin{equation}\label{eq: K: 1twovaluefunction}
\sup\left\{
|f^{(i)}(z)|:
z\in f^{-1}(\{0\}),\;
i=0,\ldots,k-1
\right\}<\infty.
\end{equation}

If there exists a set $E$ of $\left[\frac{k}{2}\right]+2$ 
distinct points in $\mathbb{C}$ such that
\begin{equation}\label{eq: K: 2twovaluefunction}
\sup\left\{
\frac{f^{\#(k)}(z)}{\varphi(|z|)^k}
:
z\in f^{-1}(E)
\right\}
<\infty,
\end{equation}
and
\begin{equation}\label{eq: K: 3twovaluefunction}
\sup\left\{
\frac{|h^{(k+1)}(z)|+|g^{(k+1)}(z)|}
{1+\left(|h^{(k)}(z)|+|g^{(k)}(z)|\right)^{k+1}}
:
z\in f^{-1}(E)
\right\}
<\infty,
\end{equation}
then $f$ is a $\varphi$-normal harmonic mapping.
\end{theorem}

The preceding Theorem~\ref{K: twovaluefunction}, estblished for 
individual harmonic mappings, naturally suggests {\it an extension 
to families of harmonic mappings by imposing the corresponding
differential bounds uniformly over the family.} This motivates
the following result for families of harmonic mappings.
\begin{theorem}\label{K: twovaluefamily}
Let $\varphi$ be a smoothly increasing function, and
$k$ be a positive integer. Suppose $\mathcal{F}$ is a family 
of sense preserving harmonic mappings in $\mathbb{D}$ such that
\begin{equation}\label{eq:K: 1twovaluefunction}
\sup\left\{
|f^{(i)}(z)|:
z\in f^{-1}(\{0\}),\, f \in \mathcal{F}, \;
i=0,\ldots,k-1
\right\}<\infty.
\end{equation}
If there exists a set $E$ of $\left[\frac{k}{2}\right]+2$
 distinct points in $\mathbb{C}$ such that
\begin{equation}\label{eq:K: 2twovaluefamily}
\sup\left\{
\frac{f^{\#(k)}(z)}{\varphi(|z|)^k}
:
z\in f^{-1}(E), \, f \in \mathcal{F}
\right\}
<\infty,
\end{equation}
and
\begin{equation}\label{eq:K: 3twovaluefamily}
\sup\left\{
\frac{|h^{(k+1)}(z)|+|g^{(k+1)}(z)|}
{1+\left(|h^{(k)}(z)|+|g^{(k)}(z)|\right)^{k+1}}
:
z\in f^{-1}(E), \, f \in \mathcal{F}
\right\}
<\infty,
\end{equation}
then $\mathcal{F}$ is a $\varphi$-normal family.
\end{theorem}
 
 Normality criteria involving spherical derivatives and
 differential expressions have been studied extensively;
 see, for example, \cite{chenlappan,Tanthin}. In view of these
results, we investigate a polynomial extension of Result~\ref{threevaluefunction}.
More precisely, we ask whether $f^\#$ can be replaced by a polynomial $P(f^\#)$
while preserving the conclusion of $\varphi$-normality. This is addressed in
the following theorem.
\begin{theorem}\label{threevaluepoly}
Let $f$ be a sense-preserving harmonic mapping in $\mathbb{D}$ and
 let $P(z)$ be a non-constant polynomial of degree $m \geq 1$ such 
 that
  the coefficients of $P(z)$ are non-negative real numbers. If there 
  exists a set $E$ of three distinct points in $\mathbb{C}$ such that
\begin{equation}\label{eq: threevaluepoly}
\sup \left\{
\frac{P\!\left(f^{\#}(z)\right)}{\varphi(|z|)}
: z \in f^{-1}(E)
\right\}
< \infty,
\end{equation}
then $f$ is a $\varphi$-normal harmonic mapping.
\end{theorem}

If we take $P(z)=z$, then $P(f^\#(z))=f^\#(z)$ then
 Result~\ref{threevaluefunction} is recovered as a special case.

\section{Some Essential lemmas}

In this section, we recall some known results and establish a
few auxiliary lemmas that will be used in the proofs of our
main theorems.
\smallskip

We first recall the following rescaling lemma, which provides a useful characterization of
$\varphi$-normality. Roughly speaking, it says that the failure of
$\varphi$-normality can be detected through a suitable rescaling around a
sequence of points, which yields a non-constant harmonic mapping on
$\mathbb{C}$ as a locally uniform limit.

\begin{lemma}\label{zalharfun}\cite[Theorem 1.3]{Bohra}
  A non-constant mapping $f$ in $\mathbb{D}$ is
   $\varphi$-normal if and only if there do not exist sequences
   $\{z_n\}\subset\mathbb{D}$, and $\rho_n >0$ with $\rho_n \to 0$ 
   as $n\to \infty$ such that
   \[
   \lim_{n\to\infty}f\left(z_n+\frac{\rho_n\zeta}{\varphi(|z_n|)}\right)=F(\zeta)
   \]
   locally uniformly in $\mathbb{C}$, where $F$ is a non-constant harmonic mapping.
  
  \end{lemma}
  
  We also recall the following normalized rescaling result for 
  $\varphi$-normal families for harmonic mappings.
\begin{lemma}\label{zalcmann}\cite[Theorem 1.11]{charakbharti}
Let $\varphi:[0,1)\to(0,\infty)$ be a continuous smoothly increasing
function and let $\beta\in(-1,\infty)$. If a family
$\mathcal{F}$ of non-constant harmonic mappings in $\mathbb{D}$ is
not $\varphi$-normal, then there exist sequences
$\{z_n\}\subset\mathbb{D}$,
$\{f_n\}\subset\mathcal{F}$ and
$\{\rho_n\}\subset(0,1)$, with
$\rho_n\to0$, such that
\[
F_n(\eta)
=
\left(
\frac{\rho_n}{\phi(|z_n|)}
\right)^{\beta}
f_n\!\left(
z_n+
\frac{\rho_n\eta}{\phi(|z_n|)}
\right)
\]
converge locally uniformly in $\mathbb{C}$ to a non-constant harmonic
mapping $F$ satisfying
\[
F^{\#}(\eta)\leq F^{\#}(0)=1.
\]
\end{lemma}

We shall also use the following Hurwitz-type result for
zeros of sense-preserving harmonic mappings.

\begin{lemma}\label{hurwitz}\cite[Para.3, p. 10 ]{Duren}
Let $\{f_n\}$ be a sequence of sense-preserving harmonic mappings in
$\mathbb{D}$ converging locally uniformly to a sense-preserving harmonic
mapping $f$. Then $z_0\in\mathbb{D}$ is a zero of $f$ if and only if
$z_0$ is a cluster point of the zeros of $f_n$, $n\geq 1$.
\end{lemma}

We need some results from the value distribution theory of 
Nevanlinna for details, see~\cite{Lo2013}.
\begin{lemma}[First Fundamental Theorem]
Let $f$ be a meromorphic mapping in $\mathbb{C}$. Then for any
$a\in\mathbb{C}$, we have
\[
T\!\left(r,\frac{1}{f-a}\right)
=
T(r,f)+O(1).
\]
\end{lemma}

We recall the second fundamental theorem in the form as needed below.
\begin{lemma}[Second Fundamental Theorem,\label{SecFunThe}]
Let $f$ be a meromorphic mapping in $\mathbb{C}$ and
$a_i$, $1\leq i\leq q$, be $q$ $(\geq 3)$ distinct values in
$\mathbb{\hat C}$. Then
\[
(q-2)T(r,f)
\leq
\sum_{i=1}^{q}
\overline{N}\!\left(r,\frac{1}{f-a_i}\right)
+S(r,f).
\]
\end{lemma}

The following multiplicity lemma is a standard consequence of the
second fundamental theorem and will be useful for
controlling the number of totally ramified values. 

\begin{lemma}\cite[Lemma 2.6]{charakbharti}
Let $f$ be a non-constant entire mapping. Then there are at most two
values $a$ for which all zeros of $f-a$ are multiple.
\end{lemma}

The next lemma transfers the preceding multiplicity property to
sense-preserving harmonic mappings.

\begin{lemma}\label{threevaluelemma}\cite[Lemma 2.7]{charakbharti}
Let $f=h+\overline{g}$ be a sense-preserving harmonic mapping in
$\mathbb{C}$ with $g(0)=0$. Then there are at most two values $a$
for which all zeros of $f-a$ are multiple.
\end{lemma}

We also need the following higher-multiplicity analogue for the 
higher-order arguments.
\begin{lemma}\label{twovaluelemma}
  Let $f$ be a non-constant entire mapping. Then there is at most 
  one value $a$ in $\mathbb{C}$ for which all zeros of $f-a$ have 
  multiplicity at least three.
\end{lemma}

\begin{proof}
We use the Second Fundamental Theorem of 
Nevanlinna~\ref{SecFunThe} to prove the lemma. Suppose, to the
contrary, that there exist two distinct values
$a_1, a_2\in\mathbb{C}$ such that all zeros of
$f-a_j$ $(j=1,2)$ have multiplicity at least $3$. Since $f$ is
entire,
\[
\overline{N}(r,f)=0,
\]
and since every zero of $f-a_j$ has multiplicity at least 3,
\[
\overline{N}\!\left(r,\frac{1}{f-a_j}\right)
\leq
\frac{1}{3}
N\!\left(r,\frac{1}{f-a_j}\right)
\leq
\frac{1}{3}T(r,f)+O(1), \quad j=1,2.
\]

Hence applying Lemma~\ref{SecFunThe} to the set 
$\{a_1,a_2\}\cup\{\infty\}$,
\[
T(r,f)
\leq
\overline{N}(r,f)
+\sum_{j=1}^{2}
\overline{N}\!\left(r,\frac{1}{f-a_j}\right)
+S(r,f)
\leq
\frac{2}{3}T(r,f)+S(r,f),
\]
so that
\[
\frac{1}{3}T(r,f)\leq S(r,f),
\]
which is impossible. Hence the proof.
\end{proof}

The preceding result extends naturally to a
sense-preserving harmonic mapping as follows.
\begin{lemma}\label{twovalueharmonic}
  Let $f=h+\overline{g}$ be a sense-preserving harmonic mapping in
$\mathbb{C}$ with $g(0)=0$. Then there is at most one value $a$
for which all zeros of $f-a$ have multiplicity at least three.
\end{lemma}

\begin{proof}
Let $f$ is sense-preserving harmonic mapping in $\mathbb{C}$,
 and $\omega(z)=g'(z)/h'(z)$, since $f$ is a sense-preserving harmonic 
mapping so $|\omega(z)|<1$ for $z\in\mathbb{C}$, 
therefore Liouville's theorem gives
$\omega(z)\equiv c$ as given in the proof~\cite[Lemma~2.5]{Bohra}, 
for any $a\in\mathbb{C}$ the equation
$f(z)=a$ is equivalent to
\begin{equation}\label{eq: hequivf}
h(z)=a^{*}:=
\frac{a-\overline{c\,a}
+\overline{c\,h(0)}-|c|^{2}h(0)}
{1-|c|^{2}}.
\end{equation}

Since by Lemma~\ref{twovaluelemma} there is at most
one value $a^{*}$ for which every zero of $h-a^{*}$ has 
multiplicity at least three, and hence there
is at most one value $a$ for which every zero of $f-a$
 has multiplicity at least three.
\end{proof}

\section{Proof of the theorem}
%proof of theorem 1.1
\begin{proof}[\bf{Proof of Theorem~\ref{Mainfunction}}]
 We argue by contradiction; suppose that $f$ is not a
$\varphi$-normal harmonic mapping. Then by Lemma~\ref{zalharfun},
there exist points
$\{z_n\}\subset\mathbb{D}$ with $|z_n|\to 1^-$ 
and positive real numbers 
$\rho_n$ with
$\rho_n\to 0$ as $n\to\infty$, such that the sequence
\begin{equation}\label{eq: zalharfun}
F_n(\zeta)
:=
f\!\left(
z_n+\frac{\rho_n\zeta}{\varphi(|z_n|)}
\right)
=
h\!\left(
z_n+\frac{\rho_n\zeta}{\varphi(|z_n|)}
\right)
+
\overline{
g\!\left(
z_n+\frac{\rho_n\zeta}{\varphi(|z_n|)}
\right),
}
\end{equation}
converges locally uniformly to
a non-constant
 sense-preserving harmonic mapping
$F(\zeta)=H(\zeta)+\overline{G(\zeta)}$ in $\mathbb{C}$.

Let
$
F_n(\zeta)=h_n(\zeta)+\overline{g_n(\zeta)},
$
where
$h_n(\zeta)
:=
h\!\left(
z_n+\dfrac{\rho_n\zeta}{\varphi(|z_n|)}
\right)$, and 
$g_n(\zeta)
:=
g\!\left(
z_n+\dfrac{\rho_n\zeta}{\varphi(|z_n|)}
\right).
$
Therefore
$
h_n^{(i)}(\zeta)\to H^{(i)}(\zeta)$,
and
$g_n^{(i)}(\zeta)\to G^{(i)}(\zeta)
$
as $n\to\infty$, consequently,

\[
F_n^{(i)}(\zeta)
=
\frac{\rho_n^{\,i}}{\varphi(|z_n|)^i}
f^{(i)}
\!\left(
z_n+\frac{\rho_n\zeta}{\varphi(|z_n|)}
\right)
\longrightarrow
F^{(i)}(\zeta)
=
H^{(i)}(\zeta)+\overline{G^{(i)}(\zeta)}.
\]
\smallskip 

\noindent
\textbf{Claim 1.}
Every zero of $F$ has multiplicity at least $k$.\smallskip

For $k=1$, the claim is immediate, since every zero has multiplicity at least $1$. 
Thus, assume that $k\ge2$.
 Let $\zeta_0\in\mathbb C$ be a zero of $F$ that is $F(\zeta_0)=0$.
Then, Lemma~\ref{hurwitz} guarantees the existence of
 a sequence
$\zeta_n\to\zeta_0$ such that, for all sufficiently large $n$
\[
F_n(\zeta_n)
=
f\!\left(
z_n+\frac{\rho_n\zeta_n}{\varphi(|z_n|)}
\right)
=0.
\]

By the hypothesis~\eqref{eq:1MainfunHyp}, there exists a
 constant $M>0$ such that
\[
\left|
f^{(i)}
\!\left(
z_n+\frac{\rho_n\zeta_n}{\varphi(|z_n|)}
\right)
\right|
\le M,
\qquad
i=1,\ldots,k-1.
\]

Therefore,
\begin{equation}\label{eq: fi0}
|F^{(i)}(\zeta_0)|
=
\lim_{n\to\infty}
|F_n^{(i)}(\zeta_n)|  \\
=
\lim_{n\to\infty}
\frac{\rho_n^{\,i}}
{\varphi(|z_n|)^i}
\left|f^{(i)}
\!\left(
z_n+\frac{\rho_n\zeta_n}{\varphi(|z_n|)}
\right)\right|     \\
\le
\lim_{n\to\infty}
M
\left(
\frac{\rho_n}{\varphi(|z_n|)}
\right)^i .
\end{equation}

Using $\varphi(|z_n|)(1-|z_n|)\geq1$, we obtain the following estimate:
\begin{equation}\label{eq: rho/varphi0}
  \frac{\rho_n}{\varphi(|z_n|)}\leq\frac{\rho_n}
{\varphi(|z_n|)(1-|z_n|)}\longrightarrow 0.
\end{equation}

It follows from \eqref{eq: fi0}, and~\eqref{eq: rho/varphi0} that

\[
|F^{(i)}(\zeta_0)|
\leq
\lim_{n\to\infty}
M
\left(
\frac{\rho_n}{\varphi(|z_n|)}
\right)^i \longrightarrow 0,
\qquad
i=1,\ldots,k-1.
\]

Therefore, $
F^{(i)}(\zeta_0)=0,
\,
i=1,\ldots,k-1,
$
which proves Claim~1.
\smallskip

\noindent
\textbf{Claim~2.}
If $\zeta \in \mathbb{C}$ satisies $F(\zeta)\in E$, then
$
H^{(k)}(\zeta)=0,
\text{and}\,
G^{(k)}(\zeta)=0.
$\smallskip

Let $a\in E$, and $\zeta_{0}\in \mathbb{C}$ such that
$F(\zeta_{0})=a$,
then by Lemma~\ref{hurwitz}, there exists a 
sequence $\{\zeta_n^*\}$ with
$\zeta_n^*\to\zeta_0$ such that for all sufficiently large $n$
\[
F_n(\zeta_n^*)
=
f\!\left(
z_n+\frac{\rho_n\zeta_n^*}{\varphi(|z_n|)}
\right)
=a, \quad \text{define}\, w_n^*
=
z_n+\frac{\rho_n\zeta_n^*}{\varphi(|z_n|)}.
\]

Since $w_n^*\in f^{-1}(E)$, so the hypothesis \eqref{eq:2MainfunHyp}
 implies that there exists
a constant $M^*>0$ such that
\begin{equation}\label{*}
  f^{\#(k)}(w_n^*)
\leq
M^*\,\varphi(|w_n^*|)^k.
\end{equation}

Using the definition of the extended spherical
 derivative~\eqref{eq:Ksphderhar} together inequality~\eqref{*}
we obtain
\[
F_n^{\#(k)}(\zeta_n^*)
=
\left(\frac{\rho_n}{\varphi(|z_n|)}\right)^k
f^{\#(k)}(w_n^*)\leq
M^*\rho_n^k
\left(
\frac{\varphi(|w_n^*|)}
{\varphi(|z_n|)}
\right)^k.
\]

Since $\rho_n\to0$ and
$\dfrac{\varphi(|w_n^*|)}{\varphi(|z_n|)}
\longrightarrow1$, it follows that
$
F_n^{\#(k)}(\zeta_n^*)\longrightarrow 0,
$ passing to the limit yields
\[
F^{\#(k)}(\zeta_0)=0.
\]

Using the definition of the extended spherical 
derivative~\eqref{eq:Ksphderhar},
we obtain
\[
|H^{(k)}(\zeta_0)|+|G^{(k)}(\zeta_0)|=0 \Longrightarrow H^{(k)}(\zeta_0)=0,
\, \text{and}\,\,
G^{(k)}(\zeta_0)=0.
\]

This completes the proof of Claim~2.
\smallskip

Now, since $F(\zeta)=H(\zeta)+\overline{G(\zeta)}$ is a 
sense-preserving harmonic mapping in $\mathbb{C}$, let 
$\omega(\zeta)=G'(\zeta)/H'(\zeta)$. Then $|\omega(\zeta)|<1$ for 
all $\zeta \in \mathbb{C}$, hence, by Liouville's theorem, 
$\omega(\zeta)\equiv c$, where $c \in \mathbb{C}$ 
and $|c|<1$. As in the proof of~\cite[Lemma 2.5]{Bohra}, for any
$a\in \mathbb{C}, \,  F(\zeta)=a$ is equivalent to
\begin{equation}\label{eq: HequiF}
  H(\zeta)=a^*= \frac{a-\overline{ca}+\overline{cH(0)}-|c|^2H(0)}{1-|c|^2}
\end{equation}

In particular, for each $a\in E$, the equation $F(\zeta)=a$ is
 equivalent to $H(\zeta)=a^*$. Thus, corresponding to the set
$ E$, we obtain a set $E^*=\{a^*: a\in E \}$. Moreover,
$H(\zeta)\in E^* \Longrightarrow H^{(k)}(\zeta)=0$. Since $F$ 
is a non-constant and sense-preserving harmonic mapping in
 $\mathbb{C}$ having zeros of multiplicity at least
$k$, we have $H^{(k)}\not\equiv 0$. Consequently,
\[
\sum_{a^* \in E^*}\overline{N}\left(r, \frac{1}{H-a^*}\right)\leq  
\overline{N}\left(r, \frac{1}{H^{(k)}}\right)\leq T(r,H^{(k)}).
\] 

Now, since $H$ is entire, $H^{(k)}$ is also entire. Hence,
$\overline{N}(r,H)= N(r,H)=
    N(r,H^{(k)})=0$. Writing $H^{(k)}=H\frac{H^{(k)}}{H}$, 
    and using the logarithmic derivative, we 
    obtain  
    \begin{eqnarray}\label{logrderiv}
    T(r,H^{(k)})&=&m(r,H^{(k)})+O(1),\nonumber \\
    &\leq& m(r,H)+m(r,\frac{H^{(k)}}{H})+O(1),\nonumber \\
    T(r,H^{(k)})&\leq& T(r,H)+S(r,H).
   \end{eqnarray}

   Applying Lemma~\ref{SecFunThe} to the $k+3$ distinct values
   in $E^*\cup \{\infty\}$, we obtain   
\begin{eqnarray*}
  (k+3-2)T(r,H) &\leq &\overline{N}(r,H) +\sum_{a^* \in E^*}
  \overline{N}\left(r,\frac{1}{H-a^*}\right) +S(r,H), \\
  (k+1)T(r,H)&\leq& T(r,H)+S(r,H),  \\
    kT(r,H) &\leq& S(r,H).
\end{eqnarray*} 

Since $k \geq 1$, the last inequality is impossible. Therefore, $f$ 
is a $\varphi$-normal harmonic mapping.
\end{proof}\medskip

%proof of theorem 1.2
\begin{proof}[\bf{Proof of Theorem~\ref{mainfamily}}]
  Suppose, on the contrary, that $\mathcal{F}$ is not
$\varphi$-normal. Then by Lemma~\ref{zalcmann}, taking 
$\beta=0$, there exist
 a sequence of sense-preserving harmonic mappings 
 $\{f_n\} \subset \mathcal{F}$, points
$\{z_n\}\subset\mathbb{D}$ and positive real numbers
$\rho_n\to0$ such that the sequence
\begin{equation}\label{eq: zalcfamily}
F_n(\zeta)
:=
f_n\!\left(
z_n+\frac{\rho_n\zeta}{\varphi(|z_n|)}
\right)
=
h_n\!\left(
z_n+\frac{\rho_n\zeta}{\varphi(|z_n|)}
\right)
+
\overline{
g_n\!\left(
z_n+\frac{\rho_n\zeta}{\varphi(|z_n|)}
\right)
}
\end{equation}
converges locally uniformly in $\mathbb{C}$ to 
a non-constant
 sense-preserving harmonic mapping $
 F(\zeta)=H(\zeta)+\overline{G(\zeta)}.
$\smallskip

As in the proof of Theorem~\ref{Mainfunction}, the derivatives
of $F_n(\zeta)=H_n(\zeta)+\overline{G_n(\zeta)}$ converge locally 
uniformly to the corresponding 
derivatives of $F=H+\overline{G}$.

\noindent \textbf{Claim 1.}
Each zero of $F$ has multiplicity at least $k$.\smallskip

For $k=1$, the claim is immediate, since every zero has multiplicity at least $1$. Thus, assume that $k\ge2$.
Let $\zeta_0\in\mathbb C$ such that $F(\zeta_0)=0$,
then, by Lemma~\ref{hurwitz},  there exists
$\zeta_n\to\zeta_0$ such that $F_n(\zeta_n)=0$ for all sufficiently large $n$. 
Hence by~\eqref{eq:1mainfamily} and~\eqref{eq: rho/varphi0} it can be easily
obtained that
\[
|F^{(i)}(\zeta_0)|
\le
\lim_{n\to\infty}
M
\left(
\frac{\rho_n}{\varphi(|z_n|)}
\right)^i
\longrightarrow 0,
\qquad
i=1,\ldots,k-1.
\]

Therefore,
$
F^{(i)}(\zeta_0)=0,
\,
i=1,\ldots,k-1,
$
which proves Claim~1.\smallskip

\noindent
\textbf{Claim 2.}
If $\zeta \in \mathbb{C}$ satisfies $F(\zeta)\in E$, then
$
H^{(k)}(\zeta)=0, \text{and}\, 
G^{(k)}(\zeta)=0.
$\smallskip

Let $a\in E$, and let $\zeta_{0}\in \mathbb{C}$ satisfies
$
F(\zeta_{0})=a
$, then
by Lemma~\ref{hurwitz}, there exists a 
sequence $\{\zeta_n^*\}$ with
$\zeta_n^*\to\zeta_0$ such that for all sufficiently large $n$
\[
F_n(\zeta_n^*)
=
f_n\!\left(
z_n+\frac{\rho_n\zeta_n^*}{\varphi(|z_n|)}
\right)
=a, \quad \text{define} \, w_n^*
=
z_n+\frac{\rho_n\zeta_n^*}{\varphi(|z_n|)}.
\]

Since $w_n^*\in f^{-1}(E)$, the hypothesis \eqref{eq:2mainfamily}
 implies that there exists
a constant $M^*>0$ such that
\begin{equation}\label{thm21}
f_n^{\#(k)}(w_n^*)
\leq
M^*\,\varphi(|w_n^*|)^k.
\end{equation}

Using the definition of the extended spherical
 derivative~\eqref{eq:Ksphderhar} and inequlity~\eqref{thm21} we obtain
\[
F_n^{\#(k)}(\zeta_n^*)
=
\left(\frac{\rho_n}{\varphi(|z_n|)}\right)^k
f_n^{\#(k)}(w_n^*)\leq
M^*\rho_n^k
\left(
\frac{\varphi(|w_n^*|)}
{\varphi(|z_n|)}
\right)^k\longrightarrow0.
\]

Passing to the limit yields
$
F^{\#(k)}(\zeta_0)=0
$, then using the definition of the extended spherical 
derivative~\eqref{eq:Ksphderhar}
it is easy to conclude that
$
H^{(k)}(\zeta_0)=0,\,
\text{and}\,\,
G^{(k)}(\zeta_0)=0.$\smallskip

This completes the proof of Claim~2.\smallskip
 
Now, proceed similarly to the proof of 
Theorem~\ref{Mainfunction}, we see that 
for each $a\in E, \, F(\zeta)=a$ is equivalent to $H(\zeta)=a^*$,
where $a^*$
 is determined by ~\eqref{eq: HequiF}.
Thus, corresponding to the set $E$, we obtain a set $E^*=\{a^*: a\in E\}$.
Moreover, $H(\zeta)\in E^*\Longrightarrow H^{(k)}(\zeta)=0$.
 Since $F$ is a non-constant and sense-preserving harmonic 
 mapping in $\mathbb{C}$
having zeros of multiplicity at least $k$, we have $H^{(k)}\not\equiv 0$.
Therefore, using~\eqref{logrderiv} and applying Lemma~\ref{SecFunThe}
 to the distinct $k+3$ values in $E^*\cup \{\infty\}$ we obtain
\[
    kT(r,H) \leq S(r,H).
\] 

This is impossible for $k\geq1$. Hence $\mathcal{F}$ is a $\varphi$
  -normal family of harmonic mappings.
\end{proof}
\medskip

%proof of theorem 1.3
\begin{proof}[\bf{Proof of theorem~\ref{twovaluefunction}}]
Assume, on the contrary, that $f$ is not a
  $\varphi$-normal harmonic mapping. Applying the rescaling
argument used in the proof of Theorem~\ref{Mainfunction}, 
and retaining the same notations, we obtain
  a sequence of rescaled harmonic mappings as in~\eqref{eq: zalharfun}.
This sequence converges locally uniformly to a non-constant sense-preserving 
harmonic mapping of the form
  \[
  F(\zeta)=H(\zeta)+\overline{G(\zeta)}.
  \]

\noindent \textbf{Claim:} For every $a\in E$, each
 zero of $F-a$ has multiplicity at least three.
\smallskip

Let $\zeta_0 \in \mathbb{C}$ be a zero 
of $F-a$, that is $F(\zeta_0)=a$.
Then Lemma~\ref{hurwitz} implies that for all sufficiently large 
$n$, there exists a sequence of points $\{\zeta_n\}$
 in $\mathbb{C}$ converging to $\zeta_0$ such that,
\[
F_n(\zeta_n):=f\!\left(z_n+\frac{\rho_n\zeta_n}
{\varphi(|z_n|)}\right)=a.
\]

Define $\omega_n:=\!\left(z_n+\frac{\rho_n\zeta_n}
{\varphi(|z_n|)}\right)$
then, by the hypothesis~\eqref{eq:1twovaluefunction},
 there exists a constant 
$M_1>0$, such that
\begin{equation}\label{11*}
  \frac{|h'(\omega_n)|+|g'(\omega_n)|}
{1+|f(\omega_n)|^2}\leq M_1\varphi(|\omega_n|),
\end{equation}
\begin{equation}\label{*1*}
{|h'(\omega_n)|+|g'(\omega_n)|}
\leq M_1\varphi(|\omega_n|)({1+|f(\omega_n)|^2})
\leq M_1\varphi(|\omega_n|)({1+\max_{a\in E}|a|^2}).
\end{equation}

Therefore using~\eqref{eq:sphderhar} and~\eqref{11*} we have,
\[
F_n^\#(\zeta_n)=\frac{|h_n'(\zeta_n)|+|g_n'(\zeta_n)|}
{1+|F(\zeta_n)|^2}=\frac{\rho_n}{\varphi(|z_n|)}
\frac{|h'(\omega_n)|+|g'(\omega_n)|}{1+|f(\omega_n)|^2
}\leq M_1\rho_n \frac{\varphi(|\omega_n|)}{\varphi(|z_n|)}.
\]

It follows from $\rho_n \to 0$ and 
$\dfrac{\varphi(|\omega_n|)}{\varphi(|z_n|)} \longrightarrow 1$
that $F_n^{\#}(\zeta_n)\rightarrow 0 $ passing to the limit yields
\begin{equation}\label{eq: H'andG'=0}
  F^{\#}(\zeta_0)=\frac{|H'(\zeta_0)|+|G'(\zeta_0)|}{1+|F(\zeta_0)|^2}=0.
\end{equation}

Therefore, we have $H'(\zeta_0)=0$, and $G'(\zeta_0)=0$. 
Hence $\zeta_0$ is a zero of order at least
 2 of $F-a$. We now show that every zero of $F-a$ 
 has multiplicity at least $3$.
\smallskip

From the hypothesis~\ref{eq:2twovaluefunction}, for 
$\omega_n\in f^{-1}(E)$ there exists a constant $M_2>0$ such that
\begin{equation}\label{***}
  \frac{|h''(\omega_n)|+|g''(\omega_n)|}
{1+(|h'(\omega_n)|+|g'(\omega_n)|)^2}\leq M_2.
\end{equation}

Assume that $M=\max\{M_1, M_2\}$, then using~\eqref{*1*} and~\eqref{***} we have
\begin{eqnarray*}
  \frac{|h_n''(\zeta_n)|+|g_n''(\zeta_n)|}
  {1+(|h_n'(\zeta_n)|+|g_n'(\zeta_n)|)^2} 
  &=& \frac{\left(\frac{\rho_n}{\varphi(|z_n|)}
  \right)^2\left(|h''(\omega_n)|
  +|g''(\omega_n)|\right)}{1+\left(\frac{\rho_n}
  {\varphi(|z_n|)}\right)^2\left(|h'(\omega_n)|+
  |g'(\omega_n|\right)^2}, \\
  &=& \frac{\left(\frac{\rho_n}{\varphi(|z_n|)}\right)^2
  \left(|h''(\omega_n)|+|g''(\omega_n)|\right)}
  {1+\left(|h'(\omega_n)|+|g'(\omega_n)|\right)^2}\!\times\!
  \frac{1+\left(|h'(\omega_n)|+|g'(\omega_n)|\right)^2}
  {1+\left(\frac{\rho_n}{\varphi(|z_n|)}\right)^2
  \left(|h'(\omega_n)|+|g'(\omega_n)|\right)^2}, \\
   &\leq& M \left(\frac{\rho_n}{\varphi(|z_n|)}\right)^2
   (1+\left(|h'(\omega_n)|+|g'(\omega_n)|\right)^2),\\
  &\leq& \rho_n^2 M \left(1+(M(1+\max_{a\in E}|a|^2))^2\left(\frac{\varphi(|\omega_n|)}
  {\varphi(|z_n|)}\right)^2\right), \\
  &\to & 0 \,\, \text{as}\,\, n \to \infty.
\end{eqnarray*}

Thus, we obtain that $|H''(\zeta_0)|+|G''(\zeta_0|=0$, and hence
 $H''(\zeta_0)=G''(\zeta_0)=0$. Therefore, every
 zero of $F-a$ has multiplicity at least 3. Since $E$ consists of 
 two distinct values, this contradicts Lemma~\ref{twovalueharmonic}.
Thus, $f$ is a $\varphi$-normal harmonic mapping.  
\end{proof}
\medskip

%proof of theorem 1.4
\begin{proof}[\bf{Proof of Theorem~\ref{twovaluefamily}}]
  Assume, by contradiction, that $\mathcal{F}$ is not $\varphi$-normal.
  Using the same approach and notations in the proof of 
  Theorem~\ref{mainfamily}, we obtain
   a sequence of rescaled harmonic mappings as in~\eqref{eq: zalcfamily}, which
  converges locally uniformly to a non-constant sense-preserving
  harmonic mapping
  \[
  F(\zeta)=H(\zeta)+\overline{G(\zeta)}).
  \] 

\noindent {\bf Claim:} If $a\in E$, then every zero of $F-a$ has 
multiplicity at least three.

 To prove the claim, let
  $\zeta_0\in \mathbb{C}$ be a zero of $F-a$. By 
  Lemma~\ref{hurwitz}, for all sufficiently large $n$ 
there exists a sequence $\zeta_n$ 
converging to $\zeta_0$ such that
  \[F_n(\zeta_n):=f_n\!\left(z_n+\frac{\rho_n\zeta_n}
{\varphi(|z_n|)}\right):=f_n(\omega_n)=a.
  \]

Using $F_n(\zeta_n)=H_n(\zeta_n)+\overline{G_n(\zeta_n)}:=
 h_n(\omega_n)+\overline{g_n(\omega_n)}$ and arguing as in the proof of 
 Theorem~\ref{twovaluefunction}, we obtain the following: \smallskip
 
\noindent (i) An equation analogous to~\eqref{eq: H'andG'=0}, which yields
$|H'(\zeta_0)|+|G'(\zeta_0)|=0$,  consequently $H'(\zeta_0)=G'(\zeta_0)=0.$

\noindent (ii) Similarly $|H''(\zeta_0)|+|G''(\zeta_0)|=0$, and hence
$H''(\zeta_0)=0$, 
and $G''(\zeta_0)=0$.\smallskip

 Therefore, every zero of $F-a$ has multiplicity at least 3. Since $E$ 
  consists of two distinct values, this contradicts Lemma~\ref{twovalueharmonic}, 
  Hence, $\mathcal{F}$ is a $\varphi$-normal family of harmonic mappings.
\end{proof}
\medskip

%proof of theorem 1.5
\begin{proof}[\bf{Proof of Theorem~\ref{K: twovaluefunction}}]
 Suppose, to the contrary, that $f$ is not a
  $\varphi$-normal harmonic mapping.
   Applying the rescaling
argument used in the proof of Theorem~\ref{Mainfunction}, 
and retaining the same notations, we obtain
  a sequence of rescaled harmonic mappings as in~\eqref{eq: zalharfun}.
This sequence converges locally uniformly to a non-constant 
sense-preserving harmonic mapping of the form
  \[
  F(\zeta)=H(\zeta)+\overline{G(\zeta)}.
  \]

Moreover, as in the proof of Theorem~\ref{Mainfunction}, the derivatives
of $F_n(\zeta)=h_n(\zeta)+\overline{g_n(\zeta)}$ converge locally 
uniformly to the corresponding 
derivatives of $F(\zeta)=H(\zeta)+\overline{G(\zeta)}$.\smallskip

\noindent {\bf Claim 1:} Every zero of $F$ has multiplicity 
at least $k$.\smallskip

This follows directly from Claim~1 in the proof of
Theorem~\ref{Mainfunction}.\smallskip

 \noindent
\textbf{Claim 2.}
f $\zeta \in \mathbb{C}$ satisfies $F(\zeta)\in E$, then
$
H^{(k)}(\zeta)=0,\,
\text{and}\,
G^{(k)}(\zeta)=0.
$\smallskip

This follows directly from Claim~2 in the proof of
Theorem~\ref{Mainfunction}.
\smallskip

\noindent {\bf Claim 3:} f $\zeta \in \mathbb{C}$ satisfies $F(\zeta)\in E$, 
then
$ H^{(k+1)}(\zeta)=0, 
\,\text{and} \, G^{(k+1)}(\zeta)=0$.

Assume that $\zeta_0$ is a zero of $F-a$, for some $a\in E$, 
then by Lemma~\ref{hurwitz}, there exists a sequence $\zeta_n \to \zeta_0$
 such that for all sufficiently large $n$,
 \[
 F_n(\zeta_n)=f\!\left(
z_n+\frac{\rho_n\zeta_n}{\varphi(|z_n|)}
\right)=a, \,\, \text{define} \,\, \omega_n:=\!\left(
z_n+\frac{\rho_n\zeta_n}{\varphi(|z_n|)}
\right).
 \]
 \smallskip
 
From the hypothesis~\eqref{eq: K: 2twovaluefunction}, there exists a constant 
$M_1>0$, such that
 \[
 f^{\#(k)}(\omega_n)=\frac{|h^{(k)}(\omega_n)|+
 |g^{(k)}(\omega_n)|}{1+|f(\omega_n)|^{(k+1)}}
 \leq M_1\varphi(|\omega_n|),
 \]
 which implies
\begin{equation}\label{*2}
  {|h^{(k)}(\omega_n)|+|g^{(k)}(\omega_n)|}\leq 
  M_1\varphi(|\omega_n|)({1+|f(\omega_n)|^{(k+1)}})
\leq M_1\varphi(|\omega_n|)({1+\max_{a\in E}|a|^{(k+1)}}).
\end{equation}

Further, from the hypothesis~\eqref{eq: K: 3twovaluefunction},
there exists $M_2>0$ such that
\begin{equation}\label{*22}
  \frac{|h^{(k+1)}(\omega_n)|+|g^{(k+1)}(\omega_n)|}
  {1+(|h^{(k)}(\omega_n)|+|g^{(k)}(\omega_n)|)^{(k+1)}}\leq M_2.
\end{equation}

Assume that $M=\max\{M_1, M_2\}$. Then using the equations ~\eqref{*2} 
and ~\eqref{*22}, we have
\begin{eqnarray*}
  \frac{|h_n^{(k+1)}(\zeta_n)|+|g_n^{(k+1)}(\zeta_n)|}
  {1+(|h_n^{(k)}(\zeta_n)|+|g_n^{(k)}(\zeta_n)|)^{(k+1)}} 
  &=& \frac{\left(\frac{\rho_n}{\varphi(|z_n|)}
  \right)^{(k+1)}\left(|h^{(k+1)}(\omega_n)|
  +|g^{(k+1)}(\omega_n)|\right)}{1+\left(\frac{\rho_n}
  {\varphi(|z_n|)}\right)^{k(k+1)}\left(|h^{(k)}(\omega_n)|+
  |g^{(k)}(\omega_n|\right)^{(k+1)}}, \\
  &=& \frac{\left(\frac{\rho_n}{\varphi(|z_n|)}\right)^{(k+1)}
  \left(|h^{(k+1)}(\omega_n)|+|g^{(k+1)}(\omega_n)|\right)}
  {1+\left(|h^{(k)}(\omega_n)|+|g^{(k)}(\omega_n)|\right)^{(k+1)}}\\ 
  &&\!\times\!
  \frac{1+\left(|h^{(k)}(\omega_n)|+|g^{(k)}(\omega_n)|\right)^{(k+1)}}
  {1+\left(\frac{\rho_n}{\varphi(|z_n|)}\right)^{k(k+1)}
  \left(|h^{(k)}(\omega_n)|+|g^{(k)}(\omega_n)|\right)^{(k+1)}}, \\
   &\leq& M \left(\frac{\rho_n}{\varphi(|z_n|)}\right)^{(k+1)}
   \left(1+\left(|h^{(k)}(\omega_n)|+|g(^{(k)}\omega_n)|\right)^{(k+1)}\right),\\
  &\leq& \rho_n^{(k+1)} M \left(1+(M(1+\max_{a\in E}
  |a|^{(k+1)}))^{(k+1)}\left(\frac{\varphi(|\omega_n|)}
  {\varphi(|z_n|)}\right)^{(k+1)}\right),
   \\
  &\to & 0 \,\, \text{as}\,\, n \to \infty.
\end{eqnarray*}

Thus, we obtian that $|H^{(k+1)}(\zeta_0)|+|G^{(k+1)}(\zeta_0|=0$, and
 hence $H^{(k+1)}(\zeta_0)=0$, and 
 $G^{(k+1)}(\zeta_0)=0$.\smallskip
 
Now, using a similar approach to the proof of Theorem~\ref{Mainfunction}, we see that 
for each $a\in E$ the equation $F(\zeta)=a$ is equivalent to $H(\zeta)=a^*$ where $a^*$
 is determined by \eqref{eq: HequiF}.  Thus  
 corresponding to the set $E$, we obtian a set $E^*=\{a^*: a\in E\}$. Moreover,
 $H(\zeta)\in E^*\Longrightarrow H^{(k)}(\zeta)=0$, and 
 $H^{(k+1)}(\zeta)=0$. Since $F$ 
is a non-constant and sense-preserving harmonic mapping in
 $\mathbb{C}$ having zeros of multiplicity at least
$k$, we have $H^{(k)}\not\equiv 0$
 Therefore, using~\eqref{logrderiv} we get
 \[
\sum_{a^* \in E^*}\overline{N}\left(r, \frac{1}{H-a^*}\right)\leq
\frac{1}{2}N\left(r,\frac{1}{H^{(k)}}\right)\\ \leq \frac{1}{2} 
T(r,H^{(k)})\leq \frac{1}{2} T(r,H)+S(r,H).\]

Since $\overline{N}(r,H)=0$ as $H$ is entire, applying 
Lemma~\ref{SecFunThe} we obtain
\begin{eqnarray*}
  \left(\left[\frac{k}{2}\right]+3-2\right)T(r,H) 
  &\leq &\overline{N}(r,H) +\sum_{a^* \in E^*}
  \overline{N}\left(r,\frac{1}{H-a^*}\right) +S(r,H), \\
  \left(\left[\frac{k}{2}\right]+1\right)T(r,H)&\leq& 
  \frac{1}{2}T(r,H)+S(r,H),  \\
  \left(\left[\frac{k}{2}\right]+\frac{1}{2}\right)
  T(r,H) &\leq& S(r,H).
\end{eqnarray*} 

Since, for $k\geq 1$ this inequality is impossible, therefore $f$ 
is $\varphi$-normal harmonic mapping.
\end{proof}
\medskip

%proof of theorem 1.6
\begin{proof}[\bf{Proof of Theorem~\ref{K: twovaluefamily}}]
  Assume to the contrary, that $\mathcal{F}$ is not $\varphi$ normal,
  Using the same approach and notations in the proof of 
  Theorem~\ref{mainfamily}, we obtain
   a sequence of rescaled harmonic mappings as in~\eqref{eq: zalcfamily}, which
  converges locally uniformly to a non-constant sense-preserving
  harmonic mapping
  \[
  F(\zeta)=H(\zeta)+\overline{G(\zeta)}).
  \] 

Moreover, as in the proof of Theorem~\ref{Mainfunction}, the derivatives
of $F_n(\zeta)=H_n(\zeta)+\overline{G_n(\zeta)}$ converge locally 
uniformly to the corresponding 
derivatives of $F(\zeta)=H(\zeta)+\overline{G(\zeta)}.$\smallskip

\noindent {\bf Claim 1:} Each zero of $F$ has multiplicity at least $k$.
\smallskip

This follows directly from Claim~1 in the proof of
Theorem~\ref{mainfamily}.
\smallskip

\noindent
\textbf{Claim 2.}
If $\zeta \in \mathbb{C}$ satisfies $F(\zeta)\in E$, then
$
H^{(k)}(\zeta)=0,
\, \text{and}\,
G^{(k)}(\zeta)=0.
$\smallskip

This follows directly from Claim~2 in the proof of
Theorem~\ref{mainfamily}.
 \smallskip

\noindent {\bf Claim 3:} If $\zeta \in \mathbb{C}$ satisfies $F(\zeta)\in E$, then
$
H^{(k+1)}(\zeta)=0$, and $G^{(k+1)}(\zeta)=0$.\smallskip

Let $\zeta_0 \in \mathbb{C}$ be a zero of $F-a$,
for some $a\in E$, 
then by Lemma~\ref{hurwitz}, there exists a sequence 
$\zeta_n \to \zeta_0$
 such that, for all sufficiently large $n$
$ F_n(\zeta_n)=a$.
 \smallskip
 
 Proceeding similarly to the proof of Claim~3 in
 Theorem~\ref{K: twovaluefunction} using $F_n(\zeta_n)=
 H_n(\zeta_n)+\overline{G_n(\zeta_n)}:=
 h_n(\omega_n)+\overline{g_n(\omega_n)}$, we obtain
 $|H^{(k+1)}(\zeta_0)|+|G^{(k+1)}(\zeta_0)|=0$, hence $H^{(k+1)}(\zeta_0)=0$, and 
 $G^{(k+1)}(\zeta_0)=0$.\smallskip

Now, using a similar approach as in the proof of Theorem~\ref{Mainfunction}, we see that 
for each $a\in E$ the equation $F(\zeta)=a$ is equivalent to $H(\zeta)=a^*$ where $a^*$
 is determined by \eqref{eq: HequiF}.  Thus  
 corresponding to the set $E$, we obtian a set $E^*=\{a^*: a\in E\}$. Moreover,
 $H(\zeta)\in E^*\Longrightarrow H^{(k)}(\zeta)=0$, and 
 $H^{(k+1)}(\zeta)=0$. Since $F$ 
is a non-constant and sense-preserving harmonic mapping in
 $\mathbb{C}$ having zeros of multiplicity at least
$k$, we have $H^{(k)}\not\equiv 0$.
 Further, proceeding similarly to the proof of the 
 Theorem~\ref{K: twovaluefunction} using~\eqref{logrderiv}
  and Lemma~\ref{SecFunThe} we 
 get
\begin{equation*}
  \left(\left[\frac{k}{2}\right]+\frac{1}{2}\right)
  T(r, H) \leq S(r, H).
\end{equation*} 

Since this inequality is impossible,
hence $\mathcal{F}$ 
is a $\varphi$-normal family of harmonic mappings.
\end{proof}

\medskip

%proof of theorem 1.7
\begin{proof}[\bf{Proof of Theorem~\ref{threevaluepoly}}]
  
Assume, to the contrary, that $f$ is not a $\varphi$-normal harmonic
mapping. Then, by Lemma~\ref{zalcmann}, there exist sequences
$\{z_n\}\subset\mathbb D$ and $\{\rho_n\}$ with $\rho_n>0,\,
\rho_n\longrightarrow0,$
such that the sequence,
\[
F_n(\zeta)
=
f\!\left(
z_n+\frac{\rho_n\zeta}{\varphi(|z_n|)}
\right)
\]
converges locally uniformly to a
non-constant sense-preserving harmonic mapping
$F(\zeta)$ in $\mathbb C$.

Choose $a\in \mathbb{C}$ such that the equation $F(\zeta)=a$
has a simple zero at some point $\zeta_0$.
Since the zero is simple, the spherical derivative of $F$
does not vanish at $\zeta_0$, that is,
$F^{\#}(\zeta_0)>0$.

By Lemma~\ref{hurwitz}, 
there exists a sequence $\{\zeta_n\}$ satisfying
$\zeta_n\rightarrow\zeta_0$
such that, for all sufficiently large $n$
\[F_n(\zeta_n):=f(w_n)=a, \quad \text{where}\, w_n=\!\left(
z_n+\frac{\rho_n\zeta_n}{\varphi(|z_n|)}
\right).\]

Since $F_n\to F$ locally uniformly,
the corresponding spherical derivatives converge locally
uniformly, 
\[
F_n^{\#}(\zeta_n)
=
\frac{\rho_n}{\varphi(|z_n|)}
f^{\#}(w_n)\longrightarrow
\left(F^{\#}(\zeta_0)\right).
\]

Now, for every integer $t\ge1$, estimate the following
\[
  \frac{\left(f^{\#}(w_n)\right)^t}
{\varphi(|w_n|)} =\frac{\varphi(|z_n|)^t}
{\rho_n^t\,\varphi(|w_n|)}
\left(F_n^{\#}(\zeta_n)\right)^t
  \geq \frac{1}
{\rho_n^t\,(1-|z_n|)^{t-1}}\left(F_n^{\#}(\zeta_n)\right)^t 
\frac{\varphi(|z_n|)}
{\varphi(|w_n|)}.
\]

Since $
\dfrac{\varphi(|z_n|)}
{\varphi(|w_n|)}
\longrightarrow1,
$
while $\rho_n(1-|z_n|)\longrightarrow 0$,
we obtain $
\dfrac{\left(f^{\#}(w_n)\right)^t}
{\varphi(|w_n|)}
\longrightarrow\infty.$\smallskip

Let
$
P(z)=a_mz^m+a_{m-1}z^{m-1}+\cdots+a_1z+a_0$ be a non-constant polynomial,
where $a_m>0$, and each coefficient $a_j$ is a
non-negative real number.\smallskip

Therefore,
\[
\frac{P\!\left(f^{\#}(w_n)\right)}
{\varphi(|w_n|)}
=
\sum_{j=1}^{m}
a_j
\frac{\left(f^{\#}(w_n)\right)^j}
{\varphi(|w_n|)}
+
\frac{a_0}{\varphi(|w_n|)}.
\]

Hence,
\[
\frac{P\!\left(f^{\#}(w_n)\right)}
{\varphi(|w_n|)}
\longrightarrow\infty
\qquad\text{as }n\to\infty.
\]

Since $F$ is a non-constant sense-preserving harmonic mapping on
$\mathbb{C}$, Lemma~\ref{threevaluelemma} asserts that
 there are at most two complex
numbers for which every zero of $F-a$ is multiple.

Consequently, if $E$ is any set of three distinct test values in $\mathbb{C}$
then at least for one value $a\in E$,
$F(\zeta)=a$ has a simple zero.

Applying the conclusion obtained above to this value of $a$, we obtain 
a sequence $\omega_n \in f^{-1}(a)$ such that
\[
\sup_{\omega_n\in f^{-1}(E)}
\frac{P\!\left(f^{\#}(\omega_n)\right)}
{\varphi(|\omega_n|)}
\longrightarrow \infty.
\]

This contradicts the hypothesis~\eqref{eq: threevaluepoly}. Therefore,
 $f$ is a $\varphi$-normal harmonic mapping.
\end{proof}

\section*{Declarations}
\noindent{\bf Acknowledgement:}\quad The second author gratefully acknowledges
 the support received from CSIR-HRDG, New Delhi, India, through the 
Junior Research Fellowship, File No. 09\,/\,0961(21525) /2025-EMR-I.
\smallskip

\noindent{\bf Author Contributions:} \quad  All authors contributed equally.
 All authors wrote the main manuscript text
 and reviewed the manuscript.
\medskip

\noindent{\bf Data Availability:} \quad  Data sharing does not apply to 
this article, as no data sets
were generated or analysed during the current study.
\medskip

\noindent{\bf Ethical Approval:} \quad No such approval is required as this 
work does not involve any such human and/or animal
studies.
\smallskip

\noindent{\bf Conflict of Interest:}\quad  There is no conflict of interest 
related to this manuscript.
\smallskip

\noindent{\bf Competing Interest:} \quad The authors declare no
 competing interests.

\end{document}